\documentclass[10pt,a4paper]{article}
\usepackage{amssymb}
\usepackage{amsmath}
\usepackage{float}
\usepackage{subcaption}
\usepackage{tikz}

\font\Bbb=msbm10 at 10 truept
\def\n{\hbox{\Bbb N}}
\def\z{\hbox{\Bbb Z}}
\def\mod{\hbox{ mod }}

\newcommand{\qed}{
  \ifmmode
   \eqno{\qedsymbol}
  \else
    \leavevmode\unskip\penalty9999 \hbox{}\nobreak\hfill\hbox{\qedsymbol}
  \fi
}
\newcommand{\qedsymbol}{\leavevmode\vrule height 1.2ex width 1.1ex depth -.1ex}
\newenvironment{proof}{\begin{trivlist}\item[\hskip
\labelsep{\bf Proof.\quad}]}
{\hfill\qed\rm\end{trivlist}}

\newtheorem{theorem}{Theorem}[section]
\newtheorem{corollary}[theorem]{Corollary}
\newtheorem{proposition}[theorem]{Proposition}
\newtheorem{lemma}[theorem]{Lemma}

\title{Completely Regular Semigroups and the Discrete Log Problem}
\author{James Renshaw\\
\small School of Mathematical Sciences\\
\small University of Southampton\\
\small Southampton, SO17 1BJ, England\\
\small Tel: +44(0)2380593673\\
\small ORCID: 0000-0002-5571-8007\\
\small  j.h.renshaw@maths.soton.ac.uk}

\begin{document}
\maketitle
\date{January 2018}
\begin{abstract}
\noindent  We consider an application to the discrete log problem using completely regular semigroups which may provide a more secure symmetric cryptosystem than the classic system based on groups. In particular we describe a scheme that would appear, for some groups, to offer protection to a standard trial multiplication attack.
{\bf keywords} Semigroup, completely regular, discrete logarithm, cryptography
{\bf Mathematics Subject Classification} 2010: 11T71, 94A60, 20M10, 20M30
\end{abstract}

\section{Introduction and Preliminaries}
We refer the reader to~\cite{howie-95} for basic results and terminology in semigroups and in particular for the necessary background in completely regular semigroups. See also~\cite{renshaw-17} for the some background in applications of semigroup actions to the discrete log problem.

\smallskip

A semigroup $S$ is called a {\em completely simple} semigroup, if $S$ has no proper ideals and if the natural partial order on the idempotents, given by
$$
e\le f\text{ if and only if }e=ef=fe,
$$
is trivial. It can be shown by Rees' Theorem (\cite[Theorem~3.2.3]{howie-95}) that a completely simple semigroup is isomorphic to what is commonly referred to as a Rees Matrix Semigroup. A semigroup $S={\cal M}[G;I,\Lambda;P]$ is called a {\em Rees Matrix Semigroup over the group $G$} if for sets $I$ and $\Lambda$,
$$
S=I\times G\times\Lambda
$$
and $P=(p_{\lambda i})$ is a $\Lambda\times I$ matrix, referred to as the {\em sandwich matrix}, with entries in the group $G$, and where multiplication is given by
$$
(i,g,\lambda)(j,h,\mu) = (i,gp_{\lambda j}h,\mu).
$$
It is worth noting that a group $G$ is an example of a completely simple semigroup in which $|I|=|\Lambda| = 1$ and $P=(1_G)_{1\times1}$.

A semigroup $S$ is called {\em completely regular} if every element of $S$ belongs to a subgroup of $S$. It can be shown (see~\cite{howie-95}) that $S$ is completely regular if and only if $S$ is a {\em semilattice of completely simple semigroups}. That is to say, $S=\dot\bigcup_{\alpha\in Y} S_\alpha$ where each $S_\alpha$ is a completely simple semigroup and $Y$ is a semilattice, and where $S_\alpha S_\beta\subseteq S_{\alpha\wedge\beta}$. We shall denote this semigroup by $S={\cal S}[Y;S_\alpha]$.

\smallskip

Suppose now that $S={\cal S}[Y;S_\alpha]$ is a completely regular semigroup in which each $S_\alpha$ is a group. Then $S$ is called a {\em semilattice of groups}. It is in fact a {\em strong semilattice of groups} (see~\cite[Theorem 4.2.1]{howie-95}) in the sense that there are {\em structure maps} $\phi^\alpha_\beta : S_\alpha\to S_\beta$ for $\alpha\ge\beta$ with the properties
\begin{enumerate}
\item $(\forall \alpha\in Y)\;\phi^\alpha_\alpha = 1_{S_\alpha}$;
\item $(\forall\alpha, \beta, \gamma \in Y)\; \phi^\beta_\gamma\circ\phi^\alpha_\beta = \phi^\alpha_\gamma$;
\item $(\forall x\in S_\alpha, y\in S_\beta)\; xy = \left(\phi^{\alpha\phantom{\beta}}_{\alpha\wedge\beta}(x)\right)\left(\phi^\beta_{\alpha\wedge\beta}(y)\right)$.
\end{enumerate}

\bigskip

Let $G=(G,\cdot)$ be a group and let $p$ be a fixed element in $G$. Define a binary operation, $\ast$, on $G$ by
$$
x\ast y = xpy, \text{ for }x,y\in G.
$$
Then it is easy to see that the system $(G,\ast,p)$ is a group, with identity $p^{-1}$ and where the inverse of $x$ is given by the element $p^{-1}\cdot x^{-1}\cdot p^{-1}$ in $G$, and where $p^{-1}$ and $x^{-1}$ are the inverses of $p$ and $x$ in the original group $(G,\cdot)$. It is also easy to see that the map $(G,\cdot)\to (G,\ast,p)$ given by $x\mapsto xp^{-1}$ is a group isomorphism.

\medskip

Let $n$ be a positive integer and let $\z_n=\{0,1,\dots, n-1\}$ be the ring of integers modulo $n$. We are interested in the multiplicative structure of $\z_n$ and aim to show that $\z_n$, under multiplication, is a completely regular semigroup. The group of units modulo $n$ will be denoted by $U_n$. Note that $|U_n| = \phi(n)$ where $\phi$ is Euler's totient function. We will usually represent the units in $U_n$ by elements from the set of least non-negative residues. So, for example, if $p$ is prime then $U_p=\{1,\ldots, p-1\}$.

\medskip

Let $m$ be a positive integer and let $p_1, \ldots, p_m$ be distinct primes and let $I=\{1,\ldots, m\}$. Let $n=\prod_{i\in I}{p_i}$ and for any non-empty subset $S\subseteq I$, let $n_S=\prod_{i\in S}{p_i}$ and denote by $\overline S=I\setminus S$, so that $n=n_{\overline S}n_S$. Define
$$
U_S = \{n_{\overline S}x : x\in U_{n_S}\},
$$
where $U_{n_S}$ is the group of units modulo $n_S$ and let $U_\emptyset = \{0\}$.

\begin{proposition}
With the notation described above, for any non-empty subset $S\subseteq I$, $U_S$ is a subgroup of the multiplicative semigroup $\z_n$, and is isomorphic to $U_{n_S}$.
Moreover
$$
\z_n = \dot\bigcup_{S\subseteq I} U_S
$$
is a strong semilattice of groups, ${\cal S}[Y;U_S]$ in which $Y$ is the boolean algebra $\cal P(I)$.
\end{proposition}
\begin{proof}
It is intuitively clear, and easy to show in any case, that $U_S\cong(U_{n_S},\ast,n_{\overline S})\cong U_{n_S}$. Notice that the identity in $U_S$ is $\left(n_{\overline S}\right)^{-1}$, the inverse of $n_{\overline S}$ in $U_{n_S}$. Let $S$ and $T$ be distinct subsets of $I$ and suppose that $y\in U_S\cap U_T$. Then
$$
z = n_{\overline S}x = n_{\overline T}y
$$
for some $x\in U_{n_S}, y\in U_{n_T}$. Given that $S$ and $T$ are distinct, we can assume, without loss of generality, that there exists $i \in \overline S$ such that $i\not\in \overline T$. But then $p_i|z$ and so $p_i|y$ which means that $i\in\overline T$, a contradiction. Hence $U_S\cap U_T=\emptyset$.

Now, let $0\ne k\in \z_n$ and let $S_k$ be the largest subset of $I$ such that for each $i\in S_k, p_i|k$. If $S_k=\emptyset$ then $k\in U_n$. Otherwise, $k=n_{S_k}x_k$ for some $x_k \in U_{n_{\overline{S_k}}}$ and so $k\in U_{\overline{S_k}}$. Consequently $\z_n = \dot\bigcup_{S\subseteq I} U_S$.

\medskip

Suppose then that $x\in U_{n_S}, y\in U_{n_T}$ so that $z=n_{\overline S}xn_{\overline T}y\in U_SU_T$. If $S\cap T=\emptyset$ then $\overline S\cup\overline T=I$ and so $n|n_{\overline S}n_{\overline T}$. Hence $z\equiv0\mod n$ and consequently $z\in U_{S\cap T}$. Otherwise notice that
\begin{itemize}
\item $U_{n_S}U_{n_T}\subseteq U_{n_{S\cap T}}$,
\item $n_{\overline S}n_{\overline T}=n_{\overline{S\cap T}}\;n_{\overline S\cap \overline T}$,
\item $n_{\overline S\cap \overline T}\in U_{n_S}\cap U_{n_T}$.
\end{itemize}
Consequently we deduce that $U_SU_T\subseteq U_{S\cap T}$ and $\z_n$ is a (strong) semilattice of groups. The structure maps (see~\cite{howie-95}) are given by $\phi^S_T:U_S\to U_T$ for $T\subseteq S\subseteq I$ 
$$
\phi^{S}_{T}(x) = \left(n_{\overline T}\right)^{-1}x
$$
where $\left(n_{\overline T}\right)^{-1}$ is the inverse of $n_{\overline T}$ in $U_{n_T}$.
\end{proof}

As a special case:
\begin{corollary}\label{zn-corollary}
Let $p$ and $q$ be distinct primes and let $n=pq$. Then the semigroup $\z_n$, of integers modulo $n$ under multiplication, is a strong semilattice of the four groups, $U_{pq}, U_p, U_q$ and $\{0\}$, in which the semilattice $Y$ is the 4-element Boolean algebra
\begin{center}
\begin{tikzpicture}[scale=0.5]
\node (max) at (0,4) {$pq$};
  \node (a) at (-2,2) {$q$};
  \node (c) at (2,2) {$p$};
  \node (min) at (0,0) {$0$};
  \draw (min) -- (c) -- (max) -- (a) -- (min);
\end{tikzpicture}
\end{center}
and the structure maps are given by
$$
\phi^{pq}_p(x) = q^{-1}x, \phi^{pq}_p(x) = p^{-1}x, \phi^p_0(x) = 0, \phi^q_0(x)=0,
$$
where $p^{-1}$ is the inverse of the element $p$ in the group $U_q$ and $q^{-1}$ is the inverse of the element $q$ in the group $U_p$.
\end{corollary}

\begin{proposition}\label{reesmatrix-proposition}
Let $S={\cal S}[Y;S_\alpha]$ be a semilattice of finite groups $S_\alpha$, in which $Y$ has a top element, $1$, say. Construct the Rees Matrix semigroup $T={\cal M}[S;I,\Lambda;P]$ over the semigroup $S$ where the entries in $P$ are all taken from the group $S_1$. Then $T$ is completely regular.
\end{proposition}
\begin{proof}
That $T$ is a semigroup is straightforward. Let $x_\alpha\in S_\alpha$ and let $(i,x_\alpha,\lambda)\in T$. Then
$$
(i,x_\alpha,\lambda)^n = (i,(x_\alpha p_{\lambda i})^{n-1}x_\alpha,\lambda).
$$
Since $x_\alpha p_{\lambda i}\in S_\alpha$ then by letting $n-1$ be a multiple of the order of the element $x_\alpha p_{\lambda i}$ in $S_\alpha$, we see that $(i,x_\alpha,\lambda)^n=(i,x_\alpha,\lambda)$. Hence the monogenic subsemigroup, $\langle(i,x_\alpha,\lambda)\rangle$, generated by $(i,x_\alpha,\lambda)$, is actually a cyclic group and so every element of $T$ lies in a subgroup of $T$ and $T$ is completely regular.

\smallskip

In fact, if $\phi^1_\alpha:S_1\to S_\alpha$ is the structure map and if we denote by $P_\alpha$ the matrix obtained from $P$ by applying this structure map to each element of $P$, then it is reasonably clear that
$$
T = {\cal S}[Y;{\cal M}[S_\alpha;I,\Lambda;P_\alpha]].
$$
\end{proof}

\begin{corollary}\label{reesmatrix-corollary}
Let $n$ be a product of distinct primes. The semigroup $T={\cal M}[\z_n;I,\Lambda;P]$ where the entries in $P$ are taken from $U_n$ is a completely regular semigroup.
\end{corollary}

\bigskip

In the classic discrete log cipher, it is normal to view the cryptosystem as a group acting freely on another group by exponentiation. In more detail, let $p$ be a prime and let $G=U_{p-1}$, the group of units of the ring $\z_{p-1}$ and let $X=U_p$ the group of units of $\z_p$. An algebraic description of the  classic discrete log cipher involves defining a free action of $G$ on $X$, $G\times X\to X$, by $(m,x) \mapsto x^m$. By Fermat's little theorem,  since $x$ is a unit modulo $p$, then $x^{p-1}\equiv1\mod p$ and since $m$ is coprime to $p-1$ then there is a positive integer $k$ such that $mk\equiv1\mod p-1$. Hence $x^{mk}\equiv x\mod p$ and so $x^{mk}=x$ in $X$. Consequently $k$ is the `decrypt' key for the `encrypt' key $m$. In practice, of course we can use $\z_p$ instead of $X$ as only $0\in \z_p\setminus X$ and $0^m=0$ in $\z_p$. Notice that $\z_p$ is a completely regular semigroup being the union of the two groups $U_p$ and $\{0\}$.

\smallskip

Now let $p$ and $q$ be distinct primes and let $n=pq$. The RSA cipher can be described algebraically in a similar way to the classic discrete log cipher, by using Euler's Theorem rather than Fermat's Little Theorem. This says that if $x\in U_n$, the group of units modulo $n$, then $x^{\phi(n)}\equiv1\mod n$, where $\phi$ is Euler's totient function. The RSA cipher is then a free action of $U_{\phi(n)}$ on $U_n$ given by $(m,x)\mapsto x^m$. By the Euclidean Algorithm, we deduce that there exists a positive integer $k$ such that $km\equiv1\mod \phi(n)$ and so $x^{mk}=x$ in $U_n$. However, we can easily extend the action of $U_{\phi(n)}$ to $\z_n$ as follows. If $km\equiv1\mod \phi(n)$ then there exists $l\in\z$ such that $km=1+l\phi(n)=1+l(p-1)(q-1)$. Hence if $x\in\z_n\setminus\{0\}$ then
$$
x^{km}=x^{1+l\phi(n)} = x^{1+l(p-1)(q-1)}\equiv x\mod p.
$$
In a similar way $x^{km}\equiv1\mod q$ and so $x^{km}\equiv x\mod n$ by the Chinese remainder theorem. 

Consequently, from Corollary~\ref{zn-corollary}, we can view the RSA cipher as an action of the group $U_{\phi(n)}$ on the completely regular semigroup
$$
\z_n\cong U_{pq}\;\dot\bigcup\; U_q\;\dot\bigcup\; U_p\;\dot\bigcup\;\{0\}.
$$

\medskip

More generally, let $n$ be a positive integer and let $X$ be a finite semigroup such that $G=U_n$, the group of units of the ring $\z_n$, acts freely on $X$ by exponentiation. Then the action $G\times X\to X$ given by $(m,x)\mapsto x^m$ is the basis of a cryptosystem, if there exists $k\in G$ such that $x^{mk}=x$. Consequently, $\langle x\rangle$, the monogenic subsemigroup generated by $x$, is in fact a cyclic group, and so $X$ is a completely regular semigroup.

\bigskip

We note at this point however, that in \cite{banin-16} the authors show that the discrete log problem over a semigroup can be reduced, in polynomial time, to the discrete log problem over a subgroup of the semigroup. Not withstanding this, we describe a scheme involving a completely regular semigroup (in fact a completely simple one) which, by hiding part of the information relating to the semigroup multiplication, seems to exclude the possibility of computing this polynomial reduction. In addition, the scheme seems to offer some protection against a standard trial multiplication attack.

\section{Completely Simple Cryptosystems}

Suppose now that $S$ is a completely simple semigroup, considered as a Rees matrix semigroup ${\cal M}[G;I,\Lambda;P]$ and suppose also that $G$ is finite, of order $r$ so that $g^r = 1$ for all $g\in G$. Define an action of $U_r$, the group of units in $\z_r$, on $S$ by $n\cdot x = x^n$, so that if $x=(i,g,\lambda)$ then  $n\cdot x = (i,(gp_{\lambda i})^{n-1}g,\lambda)$. Notice that $|U_r| = \phi(r)$.

Suppose now that $n\in U_r$ so that $n$ is coprime to $r$, and hence there exists $m\in U_r$ such that $mn\equiv1\mod r$. Then
$$
x^{mn} = (i,(gp_{\lambda i})^{mn-1}g,\lambda) = (i,(gp_{\lambda i})^{mn}p_{\lambda i}^{-1},\lambda) = (i,(gp_{\lambda i})p_{\lambda i}^{-1},\lambda) = (i,g,\lambda) = x.
$$

Consequently if we know $n$, $x^n$ and $P$, then we can compute $x^{mn}$ and so recover $x$. We can in fact compute $x^{mn}$ in an efficient manner, as we can deduce the values of $i$ and $\lambda$ from $x^n$ and so we can deduce the value of $p_{\lambda i}$. Then
$$
(gp_{\lambda i})^{mn-1}g = (gp_{\lambda i})^{mn}p_{\lambda i}^{-1} = \left(\left((gp_{\lambda i})^{n-1}g\right)p_{\lambda i}\right)^mp_{\lambda i}^{-1}.
$$
The extra work involved over the classic group based scheme, involves two extra multiplications (by $p_{\lambda i}$ and $p_{\lambda i}^{-1}$) together with the computation of $p_{\lambda i}^{-1}$.

Suppose now we know $x$, $x^n$ and $G$. Can we compute $n$ and therefore solve the discrete log problem over $S$? If we also know $P$ then we know $p_{\lambda i}$ and so $(gp_{\lambda i})^n$. Consequently, the discrete log problem in this case is equivalent to that in the classic discrete log problem over the group $G$ and we are no better off using the completely simple semigroup rather than a group. Suppose however that $P$ is kept secret and that it is hard to deduce the value of $p_{\lambda i}$ from that of $i$ and $\lambda$. We know $(gp_{\lambda i})^{n-1}g$ and we know $g$ and hence we can compute $(gp_{\lambda i})^{n-1}$ but we don't know $p_{\lambda i}$ and so can't obviously recover the classic discrete log problem from this. According to~\cite{banin-16}, the discrete log problem over a semigroup, can be reduced, in polynomial time, to the classic discrete log problem in a subgroup of $S$, namely the kernel of the element $x$. However this assumes that we can compute with the semigroup $S$ and in order to do that with a Rees Matrix Semigroup, we would require knowledge of the sandwich matrix $P$. Consequently we must include the matrix $P$ as part of our secret key.

\medskip

In this application of Rees matrix semigroups, the sets $I$ and $\Lambda$ are being used as index sets to point at the value $p_{\lambda i}\in P$, and as such we clearly don't require both of these indices. Let us therefore assume, without loss of generality, that $|\Lambda|=1$ so that $S=I\times G, P = (p_i)_{i\in I}$ with multiplication given by $(i,g)(j,h) = (i,gp_jh)$ and so $(i,g)^n = (gp_i)^{n-1}g$. We will also assume from now on that $G$ is abelian.

\subsection{Chosen plaintext attacks}
Although we keep the values of $P$ secret, if the size of $I$ is small then we can consider the following chosen plaintext attack based on the existence of an oracle for solving the classic discrete log problem over the group $G$.  Suppose that $|I|=m$ and let $g_1,\ldots, g_{m+1}$ be distinct elements of $G$. Suppose also that we encrypt the values $(i,g_i)$ as $(i,g_i^np_i^{n-1})$. By the pigeon hole principle there exists $i\ne j$ such that $p_i=p_j$ and hence
$$
(g_i^np_i^{n-1})(g_j^np_j^{n-1})^{-1} = (g_ig_j^{-1})^n.
$$
Consequently we can reduce the semigroup discrete log problem over $S$ to the group discrete log problem over $G$. However, we do not know the values of $i$ and $j$ and so have to compute this quantity for each pair $1\le i,j\le m+1$, and there are $\begin{pmatrix}m+1\\2\end{pmatrix}=O(m^2)$ of these. If $m$ is relatively small, then running $m^2$ versions of the group oracle in parallel is probably feasible and consequently we need to ensure that $m$ is sufficiently large, say comparable to the size of the group $G$.

\smallskip

This clearly imposes some issues with storing the matrix $P$. If $P$ is part of the secret key then a large size for $I$ means that, in practical terms, we must compute the entries $p_i\in P$, dynamically.

\medskip

In addition, there is another potential chosen plaintext attack. Technically the value of $p_i$ is only dependant on $i$ and not on $g$. This may cause a problem, as if we could encrypt the data $(i,g)$ and $(i,g^{-1})$ then we would obtain the values $(i,(gp_i)^{n-1}g)$ and $(i,(g^{-1}p_i)^{n-1}g^{-1})$. If, as we are assuming,  $G$ is abelian, then we can calculate $(p_i^{n-1})^2$ and hence possibly $p_i^{n-1}$. Consequently we can deduce the value of $g^n$ and so again reduce the semigroup discrete log problem to the corresponding group discrete log problem. We could avert this problem if the value of $i$ was chosen in a random fashion.

\subsection{The Proposed Completely Simple Scheme}\label{cs-cipher}

Alice wants to sent Bob a secret message. Let $G$ be a finite (abelian) group and let $I=G$. Let $n\in U_{|G|}$, the group of units mod $|G|$, and $s \in I$ be two secret keys known only to Alice and Bob. Suppose also that $f:I\times I\to G$ is a function, perhaps based on a cryptographically secure hash, whose output is uniformly distributed. We encrypt $g\in G$ as follows: choose a random value $i\in I$ and let $p_i = f(i,s)$. Clearly $f$ must have the property that it is difficult to compute $f(i,s)$ from the value of $i$ alone. In addition it should be hard to calculate $s$ given $f(i,s)$ and $i$. For example the function $f(i,j) = H(i\oplus j)$ where $H$ is a suitable hash and where $i\oplus j$ is the bitwise xor of $i$ and $j$ might suffice. Alice computes $(i,(gp_i)^{n-1}g)$ as her encrypted value of $g$ to send to Bob. Bob calculates $p_i=f(i,s)$ and $m\in U_{|G|}$ such that $mn\equiv1\mod |G|$ and then computes
$$
g = \left(\left((gp_i)^{n-1}g\right)p_i\right)^mp_i^{-1}.
$$
However, as we shall see in Section~\ref{brute-force-section} below, an attacker can't easily compute $(n,p_i)$ by trial multiplication attack alone, and as long as $p_i$ is hard to deduce from the value of $i$, and $I$ is large then the two chosen plaintext attacks detailed above would appear to be infeasible.

\medskip

One other possible chosen plaintext attack comes to mind. Suppose we encrypt the value $g$ twice. The first time we obtain the encrypted value $(i,(gp_i)^{n-1}g)=(i,g^np_i^{n-1})$ and the second time the value $(j,(gp_j)^{n-1}g)=(j,g^np_j^{n-1})$. We can then deduce the value of $(p_ip_j^{-1})^{n-1}$, but as we know neither $n$ nor $p_ip_j^{-1}$ then it is hard to see what advantage we have gained. In fact even if we could deduce the value of $n$, perhaps using a different attack or some oracle, we would still need to factorise $p_ip_j^{-1}$ to deduce that values of $p_i$ and $p_j$. But in addition, this still wouldn't allow us to deduce the value of the secret key $s$ unless the function $f$ is cryptographically insecure.

\subsection{Alternative viewpoints}
We can view this completely simple cryptosystem in two alternative ways. First, let $G=(G,\cdot)$ be a group and let $p$ be a fixed element in $G$. In the isomorphic group $(G,\ast,p)$ described above, the element $x^n$ is represented by the element $(xp)^{n-1}x$ in $(G,\cdot)$. Our cryptosystem then becomes, in effect, a classic discrete log cryptosystem over the group $(G,\ast,p)$, where $p$ is chosen in the manner outlined above. However, lack of knowledge of the parameter $p$ prohibits us from computing within this group.

\medskip

Alternatively, letting $G$ be a group and $p$ a fixed element of $G$, we can view the bijection $V_p:G\to G$, $g\mapsto gp$ as a translation or shift function in which the value of $p$ is changed for each value of $g$. If we let the classic discrete log be represented by the bijection $D_n:G\to G$, $g\mapsto g^n$ then our completely simple cryptosystem is equivalent to the bijection
$$
V_p^{-1}D_nV_p : G\to G,\; g\mapsto (gp)^{n-1}g.
$$
In other words it is a conjugate of a classic discrete log system by a simple shift, the changing of the value of the shift for each block, reminding us of the Vigen\`ere cipher, albeit with an `infinitely' long key.

\subsection{Brute Force Attack}\label{brute-force-section}
For the classic discrete log cipher over a group $G$, to compute $n$ from $g$ and $g^n$ requires, at worst, $\phi(|G|)$ computations. For a careful choice of $G$ this is $O(|G|)$. Using a completely simple semigroup is significantly more expensive as not only are there more trial multiplications to consider, but the discrete log problem over $S$ seems to offer some protection to a standard trial multiplication attack. To see this, suppose we are given $(i,g)$ and $(i,(gp_i)^{n-1}g)$. Computing  $n$ using a trial multiplication attack would consists of computing $\left(gq\right)^{m-1}g$ for $1\le m\le \phi(|G|)$ and $q\in G$ in order to find the relevant pair with $(m,q) = (n,p_i)$. In principle there are a maximum of $\phi(|G|)|G|$ such computations, which is $O(|G|^2)$. However, notice that if $\gcd(m-1,|G|)=1$ then there exists $k$ such that $k(m-1)\equiv1\mod|G|$ and so for any $x\in G, x^{k(m-1)}=x$. Consequently
$$
(gp_i)^{n-1}g = \left(g\left(\left(p_ig\right)^{k(n-1)-1}p_i\right)\right)^{m-1}g
$$
and so there is no unique pair $(m,q)=(n,p_i)$ that can be computed by a simple trial multiplication attack alone.
Notice however that we also require $\gcd(m,|G|)=1$. If $|G|$ is even then of course this is impossible and so a slightly different approach is necessary. In this case, it necessary follows that $n$ must be odd and so we let $m$ be an odd integer such that $\gcd((m-1)/2,|G|)=1$. If both $n$ and $m$ are odd then $(gp_i)^{n-1}=y^2$ for $y\in G$ and so if we let $k$ be such that $k(m-1)/2\equiv1\mod|G|$ then $y^{k(m-1)/2}=y$ and hence 
$$
(gp_i)^{n-1}g = \left(g\left(g^{-1}y^k\right)\right)^{m-1}g.
$$

Consequently, not only are there, potentially, an order of magnitude more trial multiplications to perform, there are potentially many solutions $(m,q)$ to the equation
\begin{equation}
(gp_i)^{n-1}g = (gq)^{m-1}g.\label{mimics}
\end{equation}
It seems clear therefore that some other information must be gained and used in order to execute a successful trial multiplication attack. In addition, even if we could determine the values of $n$ and $p_i$, we would still need to be able to invert the function $f:I\times I\to G$ in order to determine the value of the secret $s$.

\smallskip

We shall consider the number of solutions to~\eqref{mimics} in section~\ref{imitations} below.

\smallskip

In Figure~\ref{fig:cipher} below, we demonstrate the effects of group based encryption against semigroup based encryption. An image file, Figure~\ref{fig:cipher}(a), with an 8-bit colour depth field has been encrypted using 8-bit blocks, with a value of $p=257, e=75$. The first encryption, Figure~\ref{fig:cipher}(b), uses a standard group based encryption $x\mapsto x^e\mod p$ whilst the second, Figure~\ref{fig:cipher}(c), uses a completely simple based scheme $x\mapsto (xp_i)^{e-1}x$ with secret key $s=201$ and the function $f$ based on the SHA512 hash function. Of course, such values of the parameters are unrealistically small, and we are not advocating using these cryptosystems as block ciphers as such, but it helps to demonstrate the extra diffusion incorporated in the ciphertext by the inclusion of the random data inherent in the value of  $p_i$.

\begin{figure}[H]
\centering
\begin{subfigure}[b]{.3\linewidth}
\includegraphics[width=\linewidth]{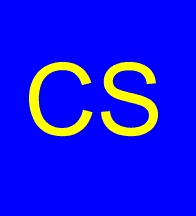}
\caption{no encryption}
\end{subfigure}
\begin{subfigure}[b]{.3\linewidth}
\includegraphics[width=\linewidth]{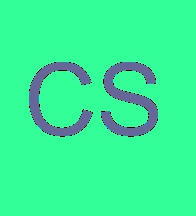}
\caption{group based cipher}
\end{subfigure}
\begin{subfigure}[b]{.3\linewidth}
\includegraphics[width=\linewidth]{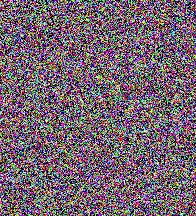}
\caption{semigroup based cipher}
\end{subfigure}
\caption{discrete log encryption on similar blocks}
\label{fig:cipher}
\end{figure}

\subsection{Completely Regular RSA cipher}

The obvious candidate for the group $G$ would be the group of units $U_p$, for $p$ a large prime, or perhaps the group associated with a suitable Elliptic Curve. However, if $n=pq$ with $p$ and $q$ large distinct primes, then Corollary~\ref{reesmatrix-corollary} implies that we can replace $G$ by $\z_n$ provided we can guarantee that the entries $p_i$ belong to $U_n$. We could then modify the RSA cipher in the following way.

Let $n=pq$ be the product of two distinct primes and let $e$ and $f$ be Bob's public and private exponents for an RSA cipher. To send a message $g$ to Bob, Alice chooses a random element $p_i\in U_n$ and, within $\z_n$, computes the pair
$$
(p_i^e,(gp_i)^{e-1}g).
$$
Bob first recovers $p_i=p_i^{ef}$ and then $g=\left(\left(\left(gp_i\right)^{e-1}g\right)p_i\right)^f$.

\smallskip

The security of this system is of course no better than that of the standard RSA cipher and has the disadvantage of resulting in twice the length of ciphertext, but it has the slight advantage of adding more diffusion by being able to encrypt two identical, non-zero, plaintext blocks into different ciphertext blocks, whilst retaining the asymmetric nature of the cryptosystem.

\section{Imitations and mimics}\label{imitations}
We have seen that brute force attacks on the completely simple system suffer from a lack of unique solutions, $(m,q)$, to the equation
\begin{equation*}
(gp_i)^{n-1}g = (gq)^{m-1}g.
\end{equation*}
In this section we show that, potentially, many such solutions exist. We initially consider a slightly more general case.

\smallskip

Let $G$ be a finite group and let $g,h,p,q\in G, n,m\in U_{|G|}$. We shall say that the triple $(m,h,q)$ {\em imitates} the triple $(n,g,p)$ if the element $h^m$ of the group $(G,\ast,q)$ coincides with the element $g^n$ of the group $(G,\ast,p)$. In other words if
$$
\left(hq\right)^{m-1}h = \left(gp\right)^{n-1}g.
$$
We wish to consider how many imitations there may be for a particular triple $(n,g,p)$.
We shall say that $h$ {\em mimics} $g$ if $(m,h,q)$ imitates $(n,g,p)$ for some $q\in G$ and some $m\in U_{|G|}$.

\begin{theorem}
Let $G$ be a finite abelian group.
\begin{enumerate}
\item If $|G|$ is odd then $h$ mimics $g$ for every $h,g\in G$.
\item If $|G|$ is even then $h$ mimics $g$ if and only if $h=gz^2$ for some $z\in G$.
\end{enumerate}
\end{theorem}
\begin{proof}
Suppose that $g,p\in G, n\in U_{|G|}$, let $x=(gp)^{n-1}g$ and let $h\in G$.
\begin{enumerate}
\item Let $k\in U_{|G|}$ be such that $k-1\in U_{|G|}$ and suppose that $m$ and $l$ are integers such that $mk\equiv1\mod|G|, l(k-1)\equiv1\mod|G|$ (see Theorem~\ref{schemmel-theorem} below for the justification that such $k$ exist).  Define $q=\left(h(x^{-1})^k\right)^l$. Then $(hq)^{m-1}h = x$.
To see this, notice that
$$
hq = h^{l+1}(x^{-1})^{kl} = h^{kl}(x^{-1})^{kl} = (hk^{-1})^{kl}.
$$
Hence
$$
(hq)^{m-1}h = (hx^{-1})^{klm-kl}h= (hx^{-1})^{l(1-k)}h = (hx^{-1})^{-1}h = x.
$$
\item Let $k\in U_{|G|}$ be such that $(k-1)/2\in U_{|G|}$ and suppose now that $m$ and $l$ are integers such that $mk\equiv1\mod|G|, l(k-1)/2\equiv1\mod|G|$ (see Theorem~\ref{totient-theorem} below for the justification that such $k$ exist). Let $h=gy^2$ for $y\in G$. Then $n$ and $k$ are both odd and so
$$
x^{-k}h = x^{-(k-1)}x^{-1}h = x^{-(k-1)}x^{-1}gy^2 = x^{1-k}(gp)^{1-n}y^2 = z^2
$$
for some $z\in G$. Now let $q=z^l$ so that
$$
q^{k-1} = (z^l)^{k-1} = (z^2)^{l(k-1)/2} = z^2 = x^{-k}h,
$$
and then
$$
hq = (xq)^k.
$$
Hence
$$
(hq)^{m-1}h = (xq)^{km-k}h = (xq)^{1-k}h = x^{1-k}h^{-1}x^kh=x.
$$
Conversely, if $x=(hq)^{m-1}h$ for some $h,q \in G, m\in U_{|G|}$ then $xq =(hq)^m$ and so $(xq)^k = hq$, where $k$ is an integer such that $km\equiv1\mod|G|$. Hence $h = x^kq^{k-1} = g(gq)^{k-1}\left((gp)^k\right)^{n-1}$, and since $k-1$ and $n-1$ are both even, the result follows.
\end{enumerate}
\end{proof}

By~\cite[Theorem~2.11]{lucido-pournaki-2008}, the number of squares in a finite group, $G$,  of even order, denoted $G^2$, satisfies $1\le G^2\le |G|-\lfloor\sqrt{|G|}\rfloor$. In the case of a group of order $2q$ with $q$ odd then $G^2=|G|/2$~\cite[Corollary~2.3]{lucido-pournaki-2008}. So, depending on the prime factorisation of $|G|$, and regardless of whether $|G|$ is odd or even, there are potentially a large number of mimics of $g\in G$.

\bigskip

We now consider when $g$ mimics $g$. Put another way, given a finite group $G$ of order $n$ and a fixed element $y\in G$, we wish to find the number of units $m$ in $Z_n$, with the property that the equation
$$
x^{m-1}=y
$$
has a solution in $G$. 

This appears to be a rather difficult problem to find an exact solution for and so we shall aim at finding a lower bound on the number of such units $m$. If $m-1$ is also a unit in $\z_n$ with inverse $k$, then $x=y^k$ is a solution. Clearly, in this case, $n$ has to be odd. On the other hand, if $n$ is even then $m$ has to be odd and so, if there is a solution, then $y=h^2$ for some $h\in G$. In this case, suppose that $(m-1)/2$ is also a unit in $\z_n$ with inverse $k$. Then $x=h^k$ is a solution.

\medskip

We consider both these cases separately, and show that the number of such imitations is potentially large.

\medskip

First, let $n$ be an odd positive integer greater than 1. We need to find the number of values of $m$ such that $\gcd(m,n)=1$ and $\gcd(m-1,n)=1$ and $1\le m\le n-1$. Let us denote this value by $S(n)$.

\begin{theorem}\label{phi-dash-prime-theorem}
Let $p$ be an odd prime and let $e$ be a positive integer. Then $S(p^e) = p^e(1-2/p)$.
\end{theorem}
\begin{proof}
We need to find those $m$ in $\{1,\ldots, p^e\}$ that are neither congruent to 0 nor 1 mod $p$.
There are $p^{e-1}$ elements in the set $\{1,\ldots, p^e\}$ that are congruent to 0 mod $p$ and $p^{e-1}$ that are congruent to 1 mod $p$. Hence $S(p^e) = p^e-2p^{e-1} = p^e(1-2/p)$.
\\
\end{proof}

\begin{theorem}
The function $S$ is multiplicative in the sense that if $m$ and $n$ are coprime and odd then $S(m)S(n) = S(mn)$.
\end{theorem}
\begin{proof}
The proof is identical to that for Euler's function $\phi(n)$ (see, for example,~\cite{jones-2005}).
\end{proof}

The following is then obvious.
\begin{theorem}\label{schemmel-theorem}
Let $n = p_1^{e_1}\ldots p_k^{e_k}$ where $p_i>2$ are distinct primes and integers $e_i\ge1$. Then
$$
S(n) = n\prod_{1\le i\le k}\left(1-\frac2{p_i}\right).
$$
\end{theorem}

\bigskip

Notice that the formula is accurate even when $n$ is even. The function $S$ is often referred to as {\em Schemmel's totient function} (See~\cite{schemmel} for more detail). Recall that
$$
\phi(n) = n\prod_{1\le i\le k}\left(1-\frac1{p_i}\right).
$$

\bigskip

Suppose now that $n$ is even and we wish to count the number of odd values of $m$ such that $\gcd(m,n)=1$ and $\gcd((m-1)/2,n)=1$. Denote this value by $T(n)$. First,
\begin{theorem}
Let $e>1$ be an integer. Then $T(2)=0$ and 
$$
T(2^e) = 2^e-3\times2^{e-2}=2^{e-2}.
$$
\end{theorem}
\begin{proof}
If $e>1$ then we need to count the odd integers $1\le m\le 2^e$ such that $m\not\equiv0\mod 2$ and $(m-1)/2\not\equiv0\mod 2$. We count the number of terms that don't satisfy this condition and subtract it from $2^e$. Hence we count those values of $m$ such that
$$
(m\equiv0\mod 2)\lor (m\equiv1\mod 4).
$$
There are $2^{e-1}$ that satisfy the first condition and $2^{e-1}/2$ that satisfy the second. Hence the result.
\end{proof}

\begin{theorem}
Let $n$ be an odd integer. Then
$$
T(2^2n)=S(n).
$$
\end{theorem}
\begin{proof}
To compute $T(4n)$ we need to remove from the set of residues $\{1,\ldots,4n\}$ numbers $m$ of the form
\begin{enumerate}
\item $m=2x$;
\item $m=xp$ where $p|n, p>1$ and $x$ is odd;
\item $m=1+2xp$ where $p|n,p>1$;
\item $m=1+4x$ where $\gcd(x,n)=1$, or $m=1$.
\end{enumerate}
Let $1\le a\le 2n$. Notice that $a$ satisfies (2) or (3) if and only if $a+2n$ does so as well. If $a$ does not satisfy (2) or (3), and if $a=1+4x$ with $\gcd(x,n)=1$ then $a+2n=3+4y$, while if $a=3+4y$ then $a+2n = 1+4x$ with $\gcd(x,n)=1$. Consequently, to compute $T(4n)$ we can remove all the even numbers, all the odd numbers greater than $2n$ and those odd numbers less than $2n$ that satisfy (2) or (3). So the only numbers left are those odd numbers less than $2n$ that do not satisfy (2) or (3).

We claim that the set of numbers left over has cardinality $S(n)$. Consider then the numbers $\{1,\ldots, n,\ n+1, \ldots, 2n\}$ and notice that to compute $S(n)$ we identify from $\{1,\ldots, n\}$ those numbers $m$ such that both $m$ and $m-1$ are coprime to $n$. But then $m+n$ also has this property in the set $\{n+1,\ldots, 2n\}$ but with a different parity. Hence $S(n)$ corresponds to the number of odd numbers in the list $\{1,\ldots, 2n\}$ that do not satisfy (2) or (3), as required.
\end{proof}

\smallskip

\begin{theorem}
Let $n$ be an odd integer and let $e>2$. Then
$$
T(2^en)=2T(2^{e-1}n)=2^{e-2}T(4n)=2^{e-2}S(n).
$$
\end{theorem}
\begin{proof}
Partition the set $\{1,\ldots, 8n\}$ into the 2 subsets $\{1,\ldots, 4n\}$ and $\{1+4n,\ldots, 8n\}$. For each element in the first subset that contributes to $T(8n)$, the element $x+4n$ in the second subset also contributes to $T(8n)$. The converse is clearly true as well. Hence $T(8n) = 2T(4n)$, and the result will now follow by induction on $e$.
\end{proof}

\bigskip

Recall some basic binomial series. Let $k$ be a positive integer and define
$$
k'=k-\frac{1+(-1)^k}2, k''=k-\frac{1-(-1)^k}2.
$$
Then
\begin{gather*}
\frac{3^k-1}2 = \binom k1 + 2\binom k2 + \ldots + 2^{k-1}\binom k k\\
2^k = 1 + \binom k1+\binom k2 + \ldots +\binom kk\\
2^{k-1}=\binom k1 + \binom k3+\ldots+\binom k{k'} = \binom k0 + \binom k2 + \ldots +\binom k{k''}\\
\end{gather*}

\medskip

The final case concerns integers of the form $2n$ for $n$ an odd positive integer. The reader may wish to have a quick look at the strategy of the proof of Theorem~\ref{twon-theorem} below, before reading further. Let $n$ be an odd positive integer and let $p>1$ be an odd divisor of $n$. Define
$$
D_p=\{m|1\le m\le 2n, 4|(m-1)\text{ and } p|m\},
$$
$$
E_p=\{m|1\le m\le 2n, 4|(m-1)\text{ and } p|(m-1)\},
$$
$$
F_p = D_p\cup E_p.
$$
\begin{lemma}\label{ap-lemma}
If $p>1$ is an odd divisor of $n$ then
$$
|E_p| = \frac{n-p}{2p}, |D_p|=\frac{n\pm p}{2p}
$$
and hence
$$
|F_p| = \frac np\text{ or }|F_p|=\frac np-1.
$$
\end{lemma}
\begin{proof}
Let $n>1$ be odd and let $p>1$ be an odd divisor of $n$.

\begin{enumerate}
\item Suppose that $1+4x=1+py\le 2n$ for positive integers $x$ and $y$. Then $y\le (2n-1)/p$ is congruent to 0 mod 4 and so there are $(2n-2p)/4p=(n-p)/2p$ different candidates for $y$. Hence $|E_p| = (n-p)/2p$.

\item Now suppose that $1+4x = py\le 2n$ for positive integers $x$ and $y$. Then $py\le 2n$ and $py$ is congruent to 1 mod 4. Suppose that $p\equiv1\mod(4)$. Then the general solution to the equation $1+4x = py$ is $1+4(x_0+pt)=p(1+4t)$ where $p=1+4x_0$ and $t\in \z$. So $1+4t\le (2n-p)/p$ and hence $t\le (2n-2p)/4p = (n-p)/2p$. Therefore there are $(n+p)/2p$ solutions in this case.

On the other hand, if $p\equiv3\mod(4)$ then the general solution is $1+4(x_0+pt)=p(3+4t)$ where $3p=1+4x_0$ and so $4t\le 2n/p-6$ and there are $(n-3p)/2p+1 = (n-p)/2p$ solutions in this case.

Hence $|D_p| = (n\pm p)/2p$.

\end{enumerate}
Consequently for each $p$ there are either $n/p$ or $n/p-1$ multiples of 4 that satisfy the relevant condition.
\end{proof}

\medskip

Notice that if $p>1$ and $q>1$ are distinct odd divisors of $n$ that are coprime then $D_{pq} = D_p\cap D_q$ and $E_{pq} = E_p\cap E_q$.
\begin{lemma}
Let $p>1$ and $q>1$ be distinct odd divisors of $n$. Then
$$
\frac{2n}{pq}-2\le |F_p\cap F_q| \le \frac{2n}{pq}+1.
$$
\end{lemma}
\begin{proof}
Suppose that $p>1$ and $q>1$ are distinct odd divisors of $n$ which are coprime. There are 4 possibilities
\begin{enumerate}
\item $1+4z=px=qy\le 2n$ for suitable $x,y,z \in\n$. Then $1+4z\in D_{pq}$ and so 
$$
|D_p\cap D_q|=|D_{pq}|=\frac{n\pm pq}{2pq}.
$$
\item $1+4z=1+px=1+qy\le2n$. Then $1+4z\in E_{pq}$ and so 
$$
|E_p\cap E_q| = |E_{pq}| = \frac{n- pq}{2pq}.
$$
\item $1+4x = 1+py = qz\le 2n$. From above the solutions to $1+4x=1+py\le 2n$ are $x=tp$ for $1\le t\le (n-p)/2p$.

We therefore need to solve $1+4pt = qz$ for $1\le t\le (n-p)/2p$. If $t_0, z_0$ is the smallest solution then the general solution is
$$
1+4pt_0+4pqs = qz_0+4pqs
$$
for $s\in\z$. This means that $t_0\le q$ and so the interval $[1,(n-p)/2p]$ can be split into $(n-pq)/2pq$ `blocks' of $q$ consecutive integers with $(n-p)/2p-(n-pq)/2p=(q-1)/2$ integers left over. Each block of $q$ integers contains exactly 1 solution and so there are either $(n-pq)/2pq$ or $(n-pq)/2pq + 1 = (n+pq)/2pq$ solutions. Hence
$$
|E_p\cap D_q|=\frac{n\pm pq}{2pq}.
$$

\item $1+4x = 1+qy=pz\le 2n$. By symmetry there are either $(n-pq)/2pq$ or $(n+pq)/2pq$ solutions. Hence
$$
|D_p\cap E_q|=\frac{n\pm pq}{2pq}.
$$
\end{enumerate}

The result then follows.
\end{proof}

\begin{lemma}
Let $p_1,\ldots, p_k$ be $k\ge 2$ distinct odd divisors of $n$. Then
$$
-2^{k-1}+\frac{2^{k-1}n}{p_1\ldots p_k}\le \left|\bigcap_{i=1}^kF_{p_i}\right|\le 2^{k-1}-1+\frac{2^{k-1}n}{p_1\ldots p_k}.
$$
\end{lemma}
\begin{proof}
Let $r=p_1\ldots p_k$. Notice that $\bigcap_{i=1}^kF_{p_i}=\bigcap_{i=1}^k{\left(D_{p_i}\cup E_{p_i}\right)}$. Hence
$$
\bigcap_{i=1}^kF_{p_i}=D_{r}\cup E_{r}\cup \bigcup_{p,q}\left(D_p\cap E_q\right)
$$
where $p$ and $q$ run through all products of fewer than $k$ of the terms $p_i$, such that all $k$ terms are used exactly once. From the previous lemma,
$$
|D_p\cap E_q| = \frac{n\pm r}{2r}
$$
and
$$
|E_r|=\frac{n-r}{2r}, |D_r|=\frac{n\pm r}{2r}.
$$
Hence
$$
 \frac{2^k}2\left(\frac{n}{r}-1\right)\le \left|\bigcap_{i=1}^kF_{p_i}\right|\le \frac{2^k-2}2\left(\frac{n}{r}+1\right)+\frac{n}{r}.
$$
\end{proof}

\bigskip

Now suppose that $n$ has $k$ distinct odd prime factors $p_1,\ldots, p_k$. Since there are $(2n-2)/4 = (n-1)/2$ values congruent to 1 mod 4, and excluding 1, in the set $\{1,\ldots,2n\}$, then there are
$$
\frac{n-1}2-\left|\bigcup_{i=1}^k F_{p_i}\right|
$$
values $m$, excluding 1, that are congruent to 1 mod 4 but such that neither $m$ nor $m-1$ has a factor in common with $n$. Using the Inclusion-Exclusion principle we see that this is bounded above by 
$$
\frac{n-1}2-\sum_i\left(\frac{n}{p_i}-1\right)+\sum_{i\ne j}\left(\frac{2n}{p_ip_j}+1\right)-\ldots
$$
which simplifies to
$$
-\frac 12+\frac n2\prod_{i=1}^k{\left(1-\frac2{p_i}\right)}+k+(2^1-1)\binom k2+2^2\binom k3+(2^3-1)\binom k 4+\ldots
$$

So the upper bound is
$$
\frac12\left(S(n)-1\right) +k+2^1\binom k2+2^2\binom k3+2^3\binom k 4+\ldots  -\binom k2-\binom k4 - \ldots
$$
or
$$
\frac12\left(S(n)-1\right) + \frac 12 3^k -\frac 12-\frac 12\left(2^k-2\right)=\frac12\left(S(n)+3^k-2^k\right).
$$

The lower bound is
$$
\frac{n-1}2-\sum_i\left(\frac{n}{p_i}\right)+\sum_{i\ne j}\left(\frac{2n}{p_ip_j}-2\right)-\ldots
$$
which simplifies to
$$
\frac12\left(S(n)-1\right)-\frac12\left(2^2\binom k2+2^3\binom k3+2^4\binom k 4+\ldots\right)+\binom k3+\binom k5+\ldots
$$
or
$$
\frac12\left(S(n)-1\right)-\frac12\left(3^k-2k-1\right)+\frac12\left(2^k-2k\right)=\frac12\left(S(n)-(3^k-2^k)\right).
$$

\begin{theorem}\label{twon-theorem}
Let $n>1$ be an odd integer with $k$ distinct prime divisors. Then

$$
\left|T(2n)-\frac{S(n)-1}2\right|\le\frac{3^k-2^k+1}2.
$$

If $n$ is prime, then $T(2n) = (n-3)/2=(S(n)-1)/2$ when $n\equiv3\mod(4)$ and $T(2n)=(n-1)/2=(S(n)+1)/2$ when $n\equiv1\mod(4)$.
\end{theorem}
\begin{proof}
To compute $T(2n)$ we need to remove from the set of residues $R=\{1,\ldots,2n\}$ numbers $m$ of the form
\begin{enumerate}
\item $m=2x$;
\item $m=xp$ where $p|n, p>1$ and $x$ is odd;
\item $m=1+2xp$ where $p|n, p>1$;
\item $m=1+4x$ where $\gcd(x,n)=1$, or $m=1$.
\end{enumerate}
Partition the set $R$ into $n$ disjoint subsets $R_i = \{i,n+i\}$ for $1\le i\le n$. Identify those sets $R_j$ for which, either $j=1$, $\gcd(j,n)\ne1$ or $\gcd(j-1,n)\ne 1$. There are $n-S(n)$ such subsets. If $j$ is odd then $j$ satisfies either (2) or (3). If $j$ is even then $n+j$ satisfies either (2) or (3). Hence all of the elements in the $R_j$ satisfy either (1), (2) or (3) and should be removed. For the subsets $R_k$ that are left, half of the elements are even and the other half are odd and do not satisfy (1), (2) or (3). There are then $S(n)$ such odd terms, some of which may satisfy (4) (notice that we have already removed $m=1$). From above, the number of these $S(n)$ terms satisfying (4) lies in the range
$$
\left[\frac12\left(S(n)-(3^k-2^k)\right),\frac12\left(S(n)+3^k-2^k\right)\right]
$$
and so $T(n)$ lies in the range
$$
\left[\frac12\left(S(n)-(3^k-2^k)\right),\frac12\left(S(n)+3^k-2^k\right)\right].
$$

The final part follows from Theorem~\ref{phi-dash-prime-theorem} and Lemma~\ref{ap-lemma}.
\end{proof}

Notice that the definition of $T(n)$ makes sense even when $n$ is odd.

\bigskip

\begin{theorem}
Let $p$ be an odd prime and $e\ge1$ an integer. Then
$$
T(p^e) = (p^e-2p^{e-1}-1)/2 = \frac{p^e}2\left(1-\frac2p-\frac1{p^e}\right)=\frac{S(p^e)-1}2.
$$
\end{theorem}
\begin{proof}
We need to count the odd integers $1\le m\le p^e$ such that $m\not\equiv0\mod p$ and $(m-1)/2\not\equiv0\mod p$. As before, we count the number of terms that don't satisfy this condition and subtract it from $p^e$. The negation of the condition is
$$
(m\equiv0\mod p)\lor (m\equiv1\mod 2p)\lor (m\equiv0\mod2).
$$
There are $p^{e-1}$ elements that satisfies the first of these, $(p^{e-1}+1)/2$ that satisfy the second and $(p^e-1)/2$ that satisfy the third. However there are $(p^{e-1}-1)/2$ that satisfy both the first and third and so the number we require is
$$
p^e-p^{e-1}-\frac{p^{e-1}+1}2-\frac{p^e-1}2+\frac{p^{e-1}-1}2 = \frac{p^e-2p^{e-1}-1}2.
$$
\end{proof}

Using techniques similar to that in Theorem~\ref{twon-theorem}, it is possible to prove the following, the details of which are omitted.

\begin{theorem}\label{phi-double-dash-theorem}
Let $n$ be an odd integer with $k$ distinct prime divisors. Then
$$
\left|T(n)- \frac{S(n)-1}2\right|\le\frac{3^k-2^{k+1}+1}{2}.
$$
\end{theorem}

\bigskip

To summarise, if we let $S(1)=1$, and if $n = 2^em$ for $e\ge 0$ and $m$ odd and if $k$ is the number of prime factors of $m$, then
\begin{theorem}\label{totient-theorem}
\begin{enumerate}
\item  For $e\ge 2$
$$T(n) = 2^{e-2}S(m).$$
\item For $e=1$
$$\left|T(n)-\frac{S(m)-1}2\right|\le\frac{3^k-2^k+1}2;$$
\begin{enumerate}
\item  if $e=1$ and $m\equiv3\mod 4$ is prime
$$T(n) = \frac{S(m)-1}2=\frac{m-3}2,$$
\item  if $e=1$ and $m\equiv1\mod 4$ is prime
$$T(n) = \frac{S(m)-1}2+1=\frac{m-1}2.$$
\end{enumerate}
\item For $e=0$
$$\left|T(n)- \frac{S(m)-1}2\right|\le\frac{3^k-2^{k+1}+1}{2};$$
\begin{enumerate}
\item When $e=0$ and $k=1,$
$$T(n) = \frac{S(n)-1}2.$$
\end{enumerate}
\end{enumerate}
\end{theorem}

\section{Summary}
Using a completely simple semigroup as outlined in Section~\ref{cs-cipher} above, to form a discrete log cryptosystem would appear to offer a more secure encryption method than the corresponding system based on groups alone. In addition, as outlined in Sections~\ref{brute-force-section} and~\ref{imitations}, there would appear to be a certain level of protection from a brute force attack. In particular, for a completely simple cryptosystem based on $\z_p$ with secret exponent $n$, and where $p$ is a large prime with the property that $p-1=2q$ with $q$ a large prime, then there are at least, $(q-3)/2 = (p-7)/4$ imitations $(m,g,q)$ for the triple $(n,g,p)$. So solving the discrete log problem by trial multiplication alone would seem to be infeasible.

\end{document}